\documentclass{article}

\usepackage{amsmath, amsthm, amsfonts, hyperref}
\usepackage[latin1]{inputenc}

\newcommand{\abs}[1]{\left\vert#1\right\vert}

\newcommand{\norm}[1]{\left\Vert#1\right\Vert}

\renewcommand{\P}{{\cal P}}
\def\R{\mathbb{R}}
\newcommand{\Po}{\P_1}            
\newcommand{\Pc}{\P_c}            

\def\tf{\tilde f}
\def\grad{\nabla}
\DeclareMathOperator{\dv}{div}
\DeclareMathOperator{\supp}{supp}

\DeclareMathOperator{\Lip}{Lip}

\def\dive{\mathrm{div}}

\def\calT{\mathcal{T}}
\def\rmL{\mathrm{L}}

\def\lip{\mathrm{Lip}}

\def\a{\alpha}
\def\b{\beta}

\newtheorem{thm}{Theorem}[section]
\newtheorem{cor}[thm]{Corollary}
\newtheorem{lem}[thm]{Lemma}
\newtheorem{prp}[thm]{Proposition}
\newtheorem{hyp}[thm]{Hypothesis}
\newtheorem{dfn}[thm]{Definition}
\newtheorem{remark}[thm]{Remark}
\newtheorem{rem}[thm]{Remark}

\begin{document}

\author{\textsc{Jos\'e A. Ca\~nizo}\thanks{Departament de Matem\`atiques,
  Universitat Aut\`onoma de Barcelona,
  08193 Bellaterra, Spain.
  E-mail: \texttt{canizo@mat.uab.es}}
  \and
  \textsc{Jos\'e A. Carrillo}\thanks{ICREA - Departament de Matem\`atiques,
  Universitat
  Aut\`onoma de Barcelona,
  08193 Bellaterra, Spain.
  E-mail: \texttt{carrillo@mat.uab.es}}
  \and
  \textsc{Jes\'us Rosado}\thanks{Departament de Matem\`atiques,
  Universitat
  Aut\`onoma de Barcelona,
  08193 Bellaterra, Spain.
  E-mail: \texttt{jrosado@mat.uab.es}}
}

\title{A well-posedness theory in measures for some kinetic models
of collective motion}

\date{April 12th 2010}

\maketitle

\begin{abstract}
  We present existence, uniqueness and continuous dependence results
  for some kinetic equations motivated by models for the collective
  behavior of large groups of individuals. Models of this kind have
  been recently proposed to study the behavior of large groups of
  animals, such as flocks of birds, swarms, or schools of fish. Our
  aim is to give a well-posedness theory for general models which
  possibly include a variety of effects: an interaction through a
  potential, such as a short-range repulsion and long-range
  attraction; a velocity-averaging effect where individuals try to
  adapt their own velocity to that of other individuals in their
  surroundings; and self-propulsion effects, which take into account
  effects on one individual that are independent of the others. We
  develop our theory in a space of measures, using mass transportation
  distances. As consequences of our theory we show also the
  convergence of particle systems to their corresponding kinetic
  equations, and the local-in-time convergence to the hydrodynamic
  limit for one of the models.
\end{abstract}

\medskip


\newpage

\section{Introduction}

The description of the collective motion (swarming) of multi-agent
aggregates resulting into large-scale structures is a striking
phenomena, as illustrated by the examples provided by birds, fish, bees or
ants. Explaining the emergence of these coordinated movements in terms
of microscopic decisions of each individual member of a swarm is a hot
matter of research in the natural sciences
\cite{camazine,couzin,parrish}. The formation of swarms and milling or
flocking patterns have been reported in animals with highly developed
social organization like insects (locusts, bees, ants, ...)
\cite{couzin}, fishes \cite{BTTYB,Bi} and birds
\cite{camazine,parrish} but also in micro-organisms as myxo-bacteria
\cite{kw}. Moreover, the understanding of natural swarms has been used
as an engineering design principle for unmanned artificial robots
operation \cite{BDT,pegoel09}.

The physics and applied mathematics literature has proliferated
and sprung in this direction in the recent years trying to model
these phenomena, mainly based on two strategies of description:
individual-based models or particle dynamics
\cite{vicek,parrish,LR,camazine,mogilner,chate,couzin,DCBC,CS1,CS2,LLE,LLE2}
and continuum models based on PDEs for the density or for the
momentum of the particle ensemble
\cite{parrish,TB04,Top06,CDMBC,EVL}. The key feature to explain is
the emergence of self-organization: flocking, milling, double
milling patterns or other coherent behavior.

Particle descriptions usually include three basic mechanisms in
different regions: short-range repulsion zone, long-range attraction
zone and alignment or orientation zone, leading to the so-called
\emph{three-zone models}. In addition, some of them incorporate a
mechanism for establishing a fixed asymptotic speed/velocity vector of
agents, as is usually observed in nature. Some of the models only
consider the orientation vector and not the speed in their discrete
version. The main differences of all these models reside in how these
three interactions are specifically considered. We will mainly work
with two generic examples in which several of the effects above are
included, namely the model for self-propelled interacting particles
introduced by D'Orsogna et al in \cite{DCBC} and the model of
alignment proposed by Cucker and Smale \cite{CS1,CS2}.

Together with particle and continuum models based on macroscopic
densities, there has been a very recent trend of mesoscopic models by
means of kinetic equations for swarming
\cite{HT08,CDP,HL08,cfrt09}. In these models one works with a
statistical description of the interacting agent system. Let us
represent by $x \in \R^d$ the position, where $d \geq 1$ stands for
the physical space dimension, and by $v \in \R^d$ the velocity. We are
interested in studying the evolution of $f = f(t,x,v)$ representing
the probability measure/density of individuals at position $x$, with
velocity $v$, and at time $t \geq 0$. These are the models we study in
the present paper. Given that we cover a variety of them, we refer the
reader to Section \ref{sec:bg} for a more detailed presentation of the
equations.

These kinetic models bridge the particle description of swarming to
the hydrodynamic one as already discussed in \cite{HT08,CDP,HL08}. The
main key idea is that solutions to particle systems are in fact
atomic-measure solutions for the kinetic equations, and solutions to
the hydrodynamic equations are solutions of a special form to the
kinetic equation; see Section \ref{sec:hydro} for more details.

In some cases, suitable compactness arguments based on the stability
properties in distances between probability measures allow to
construct a well-posedness theory for a kinetic equation. Such an
approach was done for the Vlasov equation in classical kinetic theory
\cite{neunzert,BH,dobru,spohn} with several nice reviews in
\cite{neunzert2,spohn2,Gol}. Of these references, \cite{dobru} uses
the Monge-Kantorovich-Rubinstein distance (the one we use in the
present paper); the others, as well as the recent work \cite{HL08} for
the kinetic Cucker-Smale model, use an approach based on the bounded
Lipschitz distance.

In this paper we present a generic approach to the well-posedness
of many of these models in the set of probability measures in
phase space based on the modern theory of optimal transport
\cite{Villani}. In fact, we will use the well-known
Monge-Kantorovich-Rubinstein distance between probability measures
instead of the bounded Lipschitz distance. Its better duality
properties actually make this technical approach easier in terms
of estimates leading to one of our crucial results: a stability
property of solutions to swarming equations under quite general
conditions.

We derive some consequences from this stability estimate. First, we
prove the mean-field limit, or convergence of the particle method
towards a measure solution of the kinetic equation. This mean field
limit is then established without any resorting to the BBGKY hierarchy
or the molecular chaos hypothesis \cite{Bo,CDP,HL08}. Second, we show
the stability for arbitrary times of the hydrodynamic solutions,
assuming they exist, although with constants depending on
time. Finally, the stability result can be used to obtain qualitative
properties of the measure solutions of the kinetic equations, as it
has been done in \cite{cfrt09} for the kinetic Cucker-Smale model.

This strategy is quite general, and we first demonstrate its use in a
particular kinetic model introduced in \cite{CDP} for dealing with the
mesoscopic description and certain patterns not covered by the
particle model proposed in \cite{DCBC}. Other models are treated by
the same procedure in subsequent sections, as the kinetic Cucker-Smale
model proposed in \cite{HT08} for the original alignment mechanism in
\cite{CS1,CS2}, the models studied in \cite{LLE,LLE2}, or any linear
combination of these mechanisms. We finally give general conditions on
a model that are sufficient for our well-posedness results to be
valid.

Let us comment on some limitations of the method we use. The first
one is that, as we work with solutions in a weak measure sense, we
have to require our interaction terms to be locally Lipschitz in
order in order to carry out the theory. This is a well-known
limitation in the literature for working with the mean field limit
and measure solutions, see \cite{spohn2,HJ07} and the references
therein. A less fundamental one is that we always work with
compactly supported solutions. One could probably develop a theory
substituting this condition by a suitable control on moments of
the solution, and then adapting the estimates to this setting;
however, in the present paper we do not pursue further extensions
in this direction.

Next section does a simple and brief review of the main interacting
particle systems under analysis and the needed concepts for the
Monge-Kantorovich-Rubinstein distance between probability
measures. The third section is devoted to the proof of the main result
of existence, uniqueness and stability of measure solutions to the
particular swarming equations introduced in \cite{CDP}. Section 4
generalizes this approach to a general family of these
equations. Finally, section 5 draws some consequences of the stability
property: the convergence of the particle method and the mean-field
limit are proved for the general model, while the stability in a
finite time interval of hydrodynamic solutions is shown for the
swarming model used in Section \ref{sec:swarming_model}.

\section{Preliminaries}
\label{sec:bg}

In this section we introduce the models mentioned in the
introduction. We give some particular representative cases and specify
the models to which our results apply. Also, we recall some notions
about optimal transport that shall come in handy.

\subsection{Main Kinetic Models}
The particle model proposed in \cite{DCBC} reads as:
\begin{equation*}
  \left\lbrace
    \begin{array}{ll}
      \displaystyle \frac{dx_i}{dt} = {v}_i,
      &\qquad (i = 1,\dots,N)
      \vspace{.3cm}
      \\
      \displaystyle \frac{dv_i}{dt} = (\alpha - \beta \,|{v}_i|^2) {v}_i
      - \frac1N \sum_{j \neq i } \nabla U (|x_i - x_j|),
      &\qquad (i = 1,\dots,N).
    \end{array}
  \right.
\end{equation*}
where $\alpha$, $\beta$ are nonnegative parameters,
$U:\R^d\longrightarrow \R$ is a given potential modeling the
short-range repulsion and long-range attraction typical in these
models, and $N$ is the number of particles. Here, the potential has
been scaled depending on the mass of each particle as in \cite{CDP},
where we refer for further discussion. The term corresponding to
$\alpha$ models the self-propulsion of individuals, whereas the term
corresponding to $\beta$ is the friction assumed to follow Rayleigh's
law. The balance of these two terms imposes an asymptotic speed to the
agent (if other effects are ignored), but does not influence the
orientation vector. A typical choice for $U$ is the Morse potential
which is radial and given by
$$
U(x)=k(|x|) \qquad \mbox{with} \qquad k(r) = -C_A e^{-r /\ell_A} +
C_R e^{-r /\ell_R},
$$
where $C_A, C_R$ and $\ell_A, \ell_R$ are the strengths and the
typical lengths of attraction and repulsion, respectively. This
potential does not satisfy the smoothness assumption in our main
theorems but the qualitative behavior of the particle system does not
depend on this particular fact \cite{DCBC}. In fact, a typical
potential satisfying all of our hypotheses is
$$
U(x) = -C_A e^{-|x|^2 /\ell_A^2} + C_R e^{-|x|^2 /\ell_R^2}.
$$

The kinetic equation associated to this particle model as discussed in
\cite{CDP} gives the evolution of $f = f(t,x,v)$ as
\begin{equation}
  \label{eq:swarming}
  \partial_t f + v \cdot \grad_x f
  - (\grad U * \rho) \cdot \grad_v f
  + \dv_v((\alpha - \beta \abs{v}^2) v f)
  = 0,
\end{equation}
where $\rho$ represents the macroscopic \emph{density} of $f$:
\begin{equation}
  \label{eq:def-rho}
  \rho(t,x) := \int_{\R^d} f(t,x,v) \,dv
  \quad \text{ for } t \geq 0, x \in \R^d.
\end{equation}

In the Cucker-Smale model, introduced in \cite{CS1,CS2}, the only
mechanism taken into account is the reorientation interaction between
agents. Each agent in the swarm tries to mimick other individuals by
adjusting/averaging their relative velocity with all the others.  This
averaging is weighted in such a way that closer individuals have more
influence than further ones. For a system with $N$ individuals the
Cucker-Smale model reads as
$$
\left\{\begin{array}{lr} \displaystyle\frac{dx_i}{dt} = v_i, \\[3mm]
\displaystyle\frac{dv_i}{dt} = \frac1N \displaystyle\sum_{j=1}^{N}
w_{ij} \left(v_j - v_i\right) ,
\end{array}\right.
$$
with the \emph{communication rate} $w(x)$ given by:
$$
w_{ij} = w(|x_i-x_j|)= \frac 1 {\left(1 +|x_i - x_j|^2
  \right)^\gamma}
$$
for some $\gamma \geq 0$. This particle model leads to the following
kinetic model \cite{HT08,HL08,cfrt09}:
\begin{equation}\label{eq:CS}
\frac{\partial f}{\partial t} + v\cdot \nabla_x f = \nabla_v \cdot
\left[\xi[f] \,f\right]
\end{equation}
where $\xi[f](x,v,t) = \left( H \ast f \right)(x,v,t)$, with
$H(x,v)=w(x)v$ and $\ast$ standing for the convolution in both
position and velocity ($x$ and $v$). We refer to \cite{CS1,CS2,cfrt09}
for further discussion about this model and qualitative properties.

Moreover, quite general models incorporating the three effects
previously discussed have been considered in \cite{LLE,LLE2}. In
particular, they consider that $N$ individuals follow the system:
\begin{equation}\label{eq:lle}
\left\lbrace
\begin{array}{l}
\displaystyle \frac{dx_i}{dt} = {v}_i, \vspace{.3cm}\\
\displaystyle \frac{dv_i}{dt} = F^A_i + F^I_i,
\end{array}
\right.
\end{equation}
where $F^A_i$ is the self-propulsion autonomously generated by the
$i$th-individual, while $F^I_i$ is due to interaction with the
others. The model in Section 3 corresponds to $F^A_i = (\alpha - \beta
\,|{v}_i|^2) {v}_i$, while the term $F^A_i= - \beta \, {v}_i$ is
considered in \cite{LR}, and $F^A_i= a_i - \beta \, {v}_i$ in
\cite{LLE,LLE2}. Here, $a_i$ is an autonomous self-propulsion force
generated by the $i$th-particle, and may depend on environmental
influences and the location of the particle in the school. The
interaction with other individuals can be generally modeled as:
$$
F^I_i=F^{I,x}_i + F^{I,v}_i = \sum_{j=1}^N
g_\pm(|x_i-x_j|)\frac{x_j-x_i}{|x_i-x_j|} + \sum_{j=1}^N
h_\pm(|v_i-v_j|)\frac{v_j-v_i}{|v_i-v_j|} .
$$
Here, $g_+$ and $h_+$ ($g_-$ and $h_-$) are chosen when the influence
comes from the front (behind), i.e., if $(x_j-x_i)\cdot v_i> 0$
($<0$); choosing $g_+ \neq g_-$ and $h_+ \neq h_-$ means that the
forces from particles in front and those from particles behind are
different. The sign of the functions $g_\pm(r)$ encodes the
short-range repulsion and long-range attraction for particles in front
of (+) and behind (-) the $i$th-particle. Similarly, $h_+>0$ ($<0$)
implies that the velocity-dependent force makes the velocity of
particle $i$ get closer to (away from) that of particle $j$.

In the next sections we will be concerned with the well-posedness for
measure solutions to \eqref{eq:swarming}, \eqref{eq:CS} and
generalized kinetic equations including the corresponding to the
$N$-individuals model in \eqref{eq:lle}.

\subsection{Preliminaries on mass transportation and notation}
\label{sec:transportation}

Let us recall some notation and known results about mass
transportation that we will use in the next sections. For a more
detailed approach, the interested reader can refer to
\cite{CT,Villani}.

We consider the space of probability measures $\Po(\R^d)$, consisting
of all probability measures on $\R^d$ with finite first moment. In
$\Po(\R^d)$ a natural concept of distance to work with is the
so-called \emph{Monge-Kantorovich-Rubinstein distance},
\begin{equation}
W_1(f,g) = \sup \left \{ \left |\int_{\R^d} \varphi(P)
(f(P)-g(P))\, d P \right |, \varphi \in \lip(\R^d),
\lip(\varphi)\leq 1 \right \}, \label{w1char}
\end{equation}
where $\lip(\R^d)$ denotes the set of Lipschitz functions on $\R^d$
and $\lip(\varphi)$ the Lipschitz constant of a function
$\varphi$. Denoting by $\Lambda$ the set of transference plans between
the measures $f$ and $g$, i.e., probability measures in the product
space $\R^d \times \R^d$ with first and second marginals $f$ and $g$
respectively, then we have
\begin{equation}\label{wpdef}
W_1(f, g) = \inf_{\pi\in\Lambda} \left\{ \int_{\R^d \times \R^d}
\vert P_1 - P_2 \vert \, d \pi(P_1, P_2) \right\}
\end{equation}
by Kantorovich duality. $\Po(\R^d)$ endowed with this distance is
a complete metric space. In the following proposition we recall
some of its properties. We refer to \cite{Villani} for a survey of
these basic facts.

\begin{prp}[$W_1$-properties]\label{w2properties}
The fol\-lowing properties of the distance $W_1$ hold:
\begin{enumerate}
\item[i)] {\bf Optimal transference plan:} The infimum in the
definition of the distance $W_1$ is achieved. Any joint
probability measure $\Pi_o$ satisfying:
$$
W_1(f, g) = \int_{\R^d \times \R^d} \vert P_1 - P_2 \vert \,
d\Pi_o(P_1, P_2).
$$
is called an optimal transference plan and it is generically non
unique for the $W_1$-distance.

\item[ii)] {\bf Convergence of measures:} Given $\{f_k\}_{k\ge 1}$
and $f$ in $\Po(\R^d)$, the following three assertions are
equivalent:
\begin{itemize}
\item[a)] $W_1(f_k, f)$ tends to $0$ as $n$ goes to infinity.

\item[b)] $f_k$ tends to $f$ weakly-* as measures as $k$ goes to
infinity and
$$
\sup_{k\ge 1} \int_{\vert v \vert > R} \vert v \vert \, f_k(v) \,
dv \to 0 \, \mbox{ as } \, R \to +\infty.
$$

\item[c)] $f_k$ tends to $f$ weakly-* as measures and
$$
 \int_{\R^d} \vert v \vert \, f_k(v) \, dv \to
\int_{\R^d} \vert v \vert \, f(v) \, dv  \, \mbox{ as } \,
\mbox{n} \to + \infty.
$$
\end{itemize}



\end{enumerate}
\end{prp}

Throughout the paper we will denote the integral of a function
$\varphi = \varphi(x)$ with respect to a measure $\mu$ by $\int
\varphi(x) \mu(x) \,dx$, even if the measure is not absolutely
continuous with respect to Lebesgue measure, and hence does not
have an associated density.

Given a probability measure $f \in \Po(\R^d \times \R^d)$ we
always denote by $\rho$ its first marginal, written as follows by
an abuse of notation:
\begin{equation}
  \label{eq:def-rho-abuse}
  \rho(x) := \int_{\R^d} f(x,v) \,dv.
\end{equation}
To be more precise, $\rho$ is given by its action on a
$\mathcal{C}^0_c$ function $\phi: \R^d \to \R$,
\begin{equation*}
  \int_{\R^d} \rho(x) \phi(x) \,dx
  = \int_{\R^d \times \R^d} f(x,v) \phi(x) \,dx\,dv.
\end{equation*}
For $T > 0$ and a function $f:[0,T] \to \Po(\R^d \times \R^d)$, it
is understood that $\rho$ is the function $\rho: [0,T] \to
\Po(\R^d)$ obtained by taking the first marginal at each time $t$.
Whenever we need to indicate explicitly the dependence of $\rho$
on $f$, we write $\rho[f]$ instead of just $\rho$.

We denote by $B_R$ the closed ball with center $0$ and radius $R > 0$
in the Euclidean space $\R^n$ of some dimension $n$. When we need to
explicitly indicate the dimension of the space, we will write
$B_R^n$. For a function $H : \R^n \to \R^m$, we will write $Lip_R(H)$
to denote the Lipschitz constant of $H$ in the ball $B_R \subseteq
\R^n$. For $T > 0$ and a function $H : [0,T] \times \R^n \to \R^m$, $H
= H(t,x)$, we again write $Lip_R(H)$ to denote the Lipschitz constant
\emph{with respect to $x$} of $H$ in the ball $B_R \subseteq \R^n$;
this is, $Lip_R(H)$ is the smallest constant such that
\begin{equation*}
  \abs{H(t,x_1) - H(t,x_2)} \leq Lip_R(H) \abs{x_1 - x_2}
  \quad
  \text{ for all } x_1,x_2 \in B_R, t \in [0,T].
\end{equation*}
For any such function $H$, we will denote the function depending on
$x$ at a fixed time $t$ by $H_t$; this is, $H_t(x) := H(t,x)$.

\section{Well-posedness for a system with interaction and
  self-propulsion}
\label{sec:swarming_model}

In this section we consider eq. \eqref{eq:swarming}. In this model
(and in fact, in every model considered in this paper) the total mass
is preserved, and by rescaling the equation and adapting the
parameters suitably one easily sees that we can normalize the equation
and consider only solutions with total mass $1$. We will do so and
reduce ourselves to work with probability measures.

\subsection{Notion of solution}

In order to motivate our definition of solution to equation
\eqref{eq:swarming} let us consider for a moment a general field $E$
instead of $-\nabla U * \rho$. Precisely, fix $T > 0$ and a function
$E:[0,T] \times \R^d \to \R^d$ such that:
\begin{hyp}[Conditions on $E$]
  \label{hyp:E-conditions}
  \begin{enumerate}
  \item $E$ is continuous on $[0,T] \times \R^d$,
  \item For some $C_E > 0$,
    \begin{equation}
      \label{eq:E-growth}
      \abs{E(t,x)} \leq C_E(1 + \abs{x}),
      \quad \text{ for all } t,x \in [0,T] \times \R^d,\text{ and}
    \end{equation}
  \item $E$ is \emph{locally Lipschitz with respect to $x$}, i.e., for
    any compact set $K \subseteq \R^d$ there is some $L_K > 0$ such
    that
    \begin{equation}
      \label{eq:Lipschitz-resp-x}
      \abs{E(t,x) - E(t,y)} \leq L_K \abs{x-y},
      \qquad t \in [0,T], \quad x,y \in K.
    \end{equation}
  \end{enumerate}
\end{hyp}
\noindent
We consider the equation
\begin{equation}
  \label{eq:swarming-E}
  \partial_t f + v \cdot \grad_x f
  + E \cdot \grad_v f
  + \dv_v((\alpha - \beta \abs{v}^2) v f)
  = 0,
\end{equation}
which is a linear first-order equation. The associated
characteristic system of ode's is
\begin{subequations}
  \label{eq:characteristics}
  \begin{align}
    \label{eq:characteristicsX}
    \frac{d}{dt} X &= V,
    \\
    \label{eq:characteristicsV}
    \frac{d}{dt} V &= E(t,X) + V(\alpha - \beta \abs{V}^2).
  \end{align}
\end{subequations}
\begin{lem}[Flow Map]\label{lem:exist_cont_charact}
  Take a field $E:[0,T] \times \R^d \to \R^d$ satisfying Hypothesis
  \ref{hyp:E-conditions}. Given $(X_0,V_0)\in \R^d \times \R^d$ there
  exists a unique solution $(X,V)$ to equations
  \eqref{eq:characteristicsX}-\eqref{eq:characteristicsV} in
  $\mathcal{C}^1([0,T];\R^d \times \R^d)$ satisfying $X(0)=X_0$ and
  $V(0)=V_0$. In addition, there exists a constant $C$ which depends
  only on $T$, $\abs{X_0}$, $\abs{V_0}$, $\a$, $\b$ and the constant
  $C_E$ in eq. \eqref{eq:E-growth}, such that
  \begin{equation}
    \label{eq:chars-bounded}
    \abs{ (X(t), V(t)) }
    \leq \abs{(X_0,V_0)} e^{Ct}
    \quad
    \text{ for all } t\in[0,T].
  \end{equation}
\end{lem}

\begin{proof}
  As the field $E$ satisfies the regularity and growth conditions in
  Hypothesis \ref{hyp:E-conditions}, standard results in ordinary
  differential equations show that for each initial condition
  $(X(0),V(0)) \in \R^d \times \R^d$ this system has a unique solution
  defined on $[0,T)$ (the only term in the equations which does not
  grow linearly is $-\beta V \abs{V}^2$, and it makes $\abs{V}$
  decrease, so the solution is globally defined in time). The bound
  \eqref{eq:chars-bounded} on the solutions follows from direct
  estimates on the equation, using the linear growth of the field $E$.
\end{proof}

Calling $P \equiv (X,V)$, the system
(\ref{eq:characteristicsX})-(\ref{eq:characteristicsV}) can be
conveniently written as
\begin{equation}
  \label{eq:characteristics-short}
  \frac{d}{dt} P = \Psi_E(t, P),
\end{equation}
where $\Psi_E : [0,T] \times \R^d \times \R^d \to \R^d \times \R^d$ is
the right hand side of eqs. \eqref{eq:characteristicsX},
\eqref{eq:characteristicsV}. When the field $E$ is understood we will
just write $\Psi$ instead of $\Psi_E$. Using this notation, equation
(\ref{eq:swarming-E}) can also be rewriten as
\begin{equation}
    \label{eq:swarming-E-short}
    \frac{\partial f}{\partial t} + \dive(\Psi_Ef) = 0.
\end{equation}

We can thus consider the flow at time $t \in [0,T)$ of eqs. \eqref{eq:characteristics},
\begin{equation*}
  \calT_E^t:\R^d \times \R^d \to \R^d \times \R^d.
\end{equation*}
Again by basic results in ode's, the map $(t,x,v) \mapsto
\calT_E^t(x,v)=(X,V)$ with $(X,V)$ the solution at time $t$ to
\eqref{eq:chars-bounded} with initial data $(x,v)$, is jointly
continuous in $(t,x,v)$. For a measure $f_0 \in \Po(\R^d \times
\R^d)$ it is well-known that the function
\begin{equation*}
  f: [0,T) \to \Po(\R^d\times\R^d),
  \quad
  t \mapsto f_t := \calT_E^t \# f_0
\end{equation*}
is a measure solution to eq. \eqref{eq:swarming-E}, i.e., a
solution in the distributional sense. Here we are using the mass
transportation notation of \emph{push-forward}: $f_t = \calT_E^t \#
f_0$ is defined by
\begin{equation}\label{eq:pushforward}
\int_{\R^{2d}} \zeta(x,v) \,f(t,x,v)\,d(x,v) = \int_{\R^{2d}}
\zeta(\calT_E^t(x,v)) \,f_0(x,v) \,d(x,v),
\end{equation}
for all $\zeta \in \mathcal C_b^0(\R^{2d})$. Note that in the case
where the initial condition $f_0$ is regular (say,
$\mathcal{C}^\infty_c$) this is just a way to rewrite the solution
of the equation through the method of characteristics. This
motivates the following definition:
\begin{dfn}[Notion of Solution]
  \label{dfn:solution}
  Take a potential $U \in \mathcal{C}^{1}(\R^d)$ such that
  \begin{equation}
    \label{eq:U-growth}
    \abs{ \nabla U (x) } \leq C (1+\abs{x}),
    \qquad x \in \R^d,
  \end{equation}
  for some constant $C > 0$. Take also a measure $f_0 \in \P_1(\R^d
  \times \R^d)$, and $T \in (0,\infty]$. We say that a function
  $f:[0,T] \to \P_1(\R^d \times \R^d)$ is a solution of the swarming
  equation \eqref{eq:swarming} with initial condition $f_0$ when:
  \begin{enumerate}
  \item The field $E[f] = -\nabla U * \rho$ satisfies the conditions in
    Hypothesis \ref{hyp:E-conditions}.
  \item It holds $f_t = \calT_{E[f]}^t \# f_0$.
  \end{enumerate}
\end{dfn}

\begin{rem}\label{rem:unboundE}
  This definition gives a convenient condition on $U$ so that a
  measure solution in $\P_1(\R^d \times \R^d)$ makes sense. One can
  weaken the requirement on $U$ in this definition as long as the
  requirements on $f$ are suitably strengthened (e.g., one can allow a
  faster growth of the potential if one imposes a faster decay of $f$,
  or less local regularity of $U$ if one assumes more regularity of
  $f$), but we will not consider these modifications in the present
  paper.

  Since we ask the gradient of the potential to be locally Lipschitz,
  we cannot consider potentials with a singularity at the origin. This
  is a strong limitation of the classical theory, and is considered a
  difficult problem for the mean-field limit. As for the existence
  theory, if one wants to consider more singular potentials, one can
  work with functions $f$ which are more regular than just measures,
  so that $\nabla U * \rho$ becomes locally Lipschitz and a parallel
  existence theory can be developed.
\end{rem}

\subsection{Estimates on the characteristics}

We gather in this section some estimates on solutions to the
characteristic equations \eqref{eq:characteristics}. In this section
we fix $T > 0$ and fields $E,E^1,E^2 : [0,T] \times \R^d \to \R^d$
which are assumed to satisfy Hypothesis \ref{hyp:E-conditions}, and we
consider their corresponding characteristic equations
\eqref{eq:characteristics}. Recall that $\Psi_E$ is a shorthand for
the right hand side of \eqref{eq:characteristics}, as in
\eqref{eq:characteristics-short}.

We first gather some basic regularity results for the function
which defines the right hand side of eqs.
\eqref{eq:characteristicsX}--\eqref{eq:characteristicsV}:

\begin{lem}[Regularity of the characteristic equations]
  \label{lem:reg-chareqs}
  Take a field $E: [0,T] \times \R^d \to \R^d$ which satisfies
  Hypothesis \ref{hyp:E-conditions}. Consider a number $R > 0$ and the
  closed ball $B_R \subseteq \R^d \times \R^d$.
  \begin{enumerate}
  \item $\Psi_E$ is bounded in compact sets: For $P = (X, V) \in B_R$
    and $t \in [0,T]$,
    \begin{equation*}
      \abs{\Psi_E(t,P)}
      \leq
      C
    \end{equation*}
    for some $C > 0$ which depends only on $\alpha$, $\beta$, $R$, and
    $\norm{E}_{L^\infty([0,T]\times B_R)}$.
  \item $\Psi_E$ is locally Lipschitz with respect to $x,v$: For all
    $P_1 = (X_1, V_1)$, $P_2 = (X_2, V_2)$ in $B_R$, and $t \in
    [0,T]$,
  \begin{equation*}
    \abs{\Psi_E(t,P_1) - \Psi_E(t,P_2)}
    \leq
    C (1+Lip_R(E_t)) \abs{P_1 - P_2},
  \end{equation*}
  for some number $C > 0$ which depends only on $\alpha$ and $\beta$.
  \end{enumerate}
\end{lem}

\begin{proof}
  This can be obtained by a direct calculation from
  eqs. \eqref{eq:characteristicsX}--\eqref{eq:characteristicsV}.
\end{proof}

\begin{lem}[Dependence of the characteristic equations on $E$]
  \label{lem:dep-chareqs-E}
  Take two fields $E^1, E^2 : [0,T] \times \R^d \to \R^d$ satisfying
  Hypothesis \ref{hyp:E-conditions}, and consider the functions
  $\Psi_{E^1}$, $\Psi_{E^2}$ which define the characteristic equations
  \eqref{eq:characteristics} as in
  eq. \eqref{eq:characteristics-short}. Then, for any compact (in
  fact, any measurable) set $B$,
  \begin{equation*}
    \norm{\Psi_{E^1} - \Psi_{E^2}}_{L^\infty(B)}
    \leq
    \norm{E^1 - E^2}_{L^\infty(B)}.
  \end{equation*}
\end{lem}

\begin{proof}
  Trivial from the expression of $\Psi_{E^1}$, $\Psi_{E^2}$.
\end{proof}

Now we explicitly state some results which give a quantitative
bound on the regularity of the flow $\calT_E^t$, and its dependence on
the field $E$.

\begin{lem}[Dependence of characteristics on $E$]
  \label{lem:dep-chars-E}
  Take two fields $E^1, E^2 : [0,T] \times \R^d \to \R^d$ satisfying
  Hypothesis \ref{hyp:E-conditions}, and a point $P^0 \in \R^d \times
  \R^d$. Take $R > 0$, and assume that
  \begin{equation*}
    \abs{\calT_{E^1}^t(P^0)} \leq R,
    \quad
    \abs{\calT_{E^2}^t(P^0)} \leq R
    \qquad \mbox{for }t \in [0,T].
  \end{equation*}
Then for $t \in [0,T]$ it holds that
  for some constant $C$ which depends only on $\alpha$, $\beta$, $R$
  and $\Lip_R(E^1)$
  \begin{equation}
    \label{eq:dep-chars-E}
    \abs{\calT_{E^1}^t(P^0) - \calT_{E^2}^t(P^0)}
    \leq
    \frac{e^{C t} - 1}{C}
    \sup_{s \in [0,T)} \norm{E^1_s - E^2_s}_{L^\infty(B_R)}.
  \end{equation}
\end{lem}

\begin{proof}
  For ease of notation, write $P_i(t) \equiv \calT_{E_i}^t(P^0) \equiv
  (X_i(t), V_i(t))$, for $i = 1,2$, $t \in [0,T]$. These functions
  satisfy the characteristic equations \eqref{eq:characteristics}:
  \begin{gather*}
    \frac{d}{dt} P_i = \Psi_{E_i}(t,P_i), \quad P_i(0) = P^0,
    \quad \text{ for } i = 1,2 .
  \end{gather*}
  Then, for $t \in [0,T]$, and using Lemmas \ref{lem:reg-chareqs} and
  \ref{lem:dep-chareqs-E},
  \begin{align*}
    \abs{P_1(t) - P_2(t)}
    \leq &\,
    \int_0^t \abs{\Psi_{E^1}(s,P_1(s)) - \Psi_{E^2}(s,P_2(s))} \,ds
    \\
    \leq &\,
    \int_0^t \abs{\Psi_{E^1}(s,P_1(s)) - \Psi_{E^1}(s,P_2(s))} \,ds
    \\
    &+
    \int_0^t \abs{\Psi_{E^1}(s,P_2(s)) - \Psi_{E^2}(s,P_2(s))} \,ds
    \\
    \leq &\,
    C \int_0^t \abs{P_1(s) - P_2(s)} \,ds
    +
    \int_0^t \norm{E^1_s - E^2_s}_{L^\infty(B_R)} \,ds
  \end{align*}
  where $C$ is the constant in point 2 of Lemma \ref{lem:reg-chareqs},
  which depends on $\alpha$, $\beta$, $R$ and the Lipschitz constant
  of $E^1$ with respect to $x$ in the ball $B_R$. By Gronwall's Lemma,
  \begin{align*}
    \abs{P_1(t) - P_2(t)}
    \leq &\,
    \int_0^t e^{C(t-s)} \norm{E^1_s - E^2_s}_{L^\infty(B_R)} \,ds.
    \\
    \leq &\,
    \frac{e^{C t} - 1}{C}
    \sup_{s \in (0,T)} \norm{E^1_s - E^2_s}_{L^\infty(B_R)},
  \end{align*}
  which finishes the proof.
\end{proof}

\begin{lem}[Regularity of characteristics with respect to initial
  conditions]
  \label{lem:reg-chars}
  Take $T > 0$ and a field $E: [0,T] \times \R^d \to \R^d$ satisfying
  Hypothesis \ref{hyp:E-conditions}. Take also $P_1, P_2 \in \R^d\times\R^d$
  and $R > 0$, and assume that
  \begin{equation*}
    \abs{\calT_{E}^t(P_1)} \leq R,
    \quad
    \abs{\calT_{E}^t(P_2)} \leq R
    \qquad t \in [0,T].
  \end{equation*}
  Then it holds that
  \begin{equation}
    \abs{\calT_{E}^t(P_1) - \calT_{E}^t(P_2)}
    \leq
    \abs{P_1 - P_2} e^{C \int_0^t (\Lip_R(E_s)+1) \,ds},
    \quad t \in [0,T],
  \end{equation}
  for some constant $C$ which depends only on $R$, $\alpha$ and
  $\beta$. Said otherwise, $\calT_E^t$ is Lipschitz on $B_R \subseteq \R^d
  \times \R^d$, with constant
  \begin{equation*}
    \Lip_R(\calT_E^t) \leq e^{C \int_0^t (\Lip_R(E_s)+1) \,ds},
    \quad t \in [0,T].
  \end{equation*}
\end{lem}

\begin{proof}
  Write $P_i(t) \equiv \calT_{E}^t(P_i) \equiv (X_i(t), V_i(t))$, for
  $i = 1,2$, $t \in [0,T]$.  These functions satisfy the
  characteristic equations \eqref{eq:characteristics}:
  \begin{gather*}
    \frac{d}{dt} P_i = \Psi_{E}(t,P_i), \quad P_i(0) = P_i,    \quad \text{ for } i = 1,2.
  \end{gather*}
  For $t \in [0,T]$, using Lemma \ref{lem:reg-chareqs},
  \begin{align*}
    \abs{P_1(t) - P_2(t)}
    \leq &\,
    \abs{P_1 - P_2}
    +
    \int_0^t \abs{\Psi_{E}(s,P_1(s)) - \Psi_{E}(s,P_2(s))} \,ds
    \\
    \leq &\,
    \abs{P_1 - P_2}
    +
    C \int_0^t (\Lip_R(E_s)+1) \abs{P_1(s) - P_2(s)} \,ds
  \end{align*}
  We get our result by applying Gronwall's Lemma to this inequality.
\end{proof}

\begin{lem}[Regularity of characteristics with respect to time]
  \label{lem:reg-chars-time}
  Take $T > 0$ and a field $E: [0,T] \times \R^d \to \R^d$ satisfying
  Hypothesis \ref{hyp:E-conditions}. Take $P^0 \in \R^d \times
  \R^d$, $R > 0$ and assume that
  \begin{equation*}
    \abs{\calT_{E}^t(P^0)} \leq R,
    \qquad t \in [0,T].
  \end{equation*}
  Then it holds that
  \begin{equation}
    \abs{\calT_{E}^t(P^0) - \calT_{E}^s(P^0)}
    \leq
    C \abs{t-s}
    \quad \text{ for } s, t \in [0,T],
  \end{equation}
  for some constant $C$ which depends only on $\alpha$, $\beta$, $R$
  and $\norm{E}_{L^\infty([0,T] \times B_R)}$.
\end{lem}

\begin{proof}
  Direct by definition
  of 
  $\calT_E^t(P^0)$ and from point 1 of Lemma
  \ref{lem:reg-chareqs},
  as we are assuming that $\calT_E^t(P^0)$ remains on a certain
  compact subset of $\R^d \times
  \R^d$.
\end{proof}

\subsection{Existence and uniqueness}

\begin{thm}[Existence and uniqueness of measure solutions]
  \label{thm:Existence}
  Take a potential $U \in \mathcal{C}^1(\R^d)$ such that $\nabla U$ is locally Lipschitz and such that for some $C >
  0$,
  \begin{equation}
    \label{eq:U-growth2}
    \abs{\nabla U(x)} \leq C(1+\abs{x})
    \quad \text{ for all } x \in \R^d ,
  \end{equation}
  and $f_0\in \Po(\R^d \times \R^d)$ with compact
  support. There exists a solution $f$ on $[0,+\infty)$ to equation
  \eqref{eq:swarming} with initial condition $f_0$ in the sense of
  Definition \ref{dfn:solution}. In addition,
  \begin{equation}
    \label{eq:f-continuous}
    f \in \mathcal{C}([0,+\infty); \P_1(\R^d\times \R^d))
  \end{equation}
  and there is some increasing function $R = R(T)$ such that for all $T
  > 0$,
  \begin{equation}
    \label{eq:supp-f}
    \supp f_t \subseteq B_{R(T)} \subseteq \R^d \times \R^d
    \quad \text{ for all } t \in [0,T].
  \end{equation}
  This solution is unique among the family of solutions satisfying
  \eqref{eq:f-continuous} and \eqref{eq:supp-f}.
\end{thm}

The rest of this section is dedicated to the proof of this result,
for which we will need some previous lemmas. We begin with a
general result on the transportation of a measure by two different
functions:
\begin{lem}
  \label{lem:same-f0}
  Let $P_1,  P_2 : \R^d \to \R^d$ be two Borel measurable functions. Also, take $f \in \P_1(\R^d)$. Then,
  \begin{equation}
    \label{eq:same-f0}
    W_1(P_1 \# f, P_2 \# f)
    \leq
    \norm{P_1 - P_2}_{L^\infty(\supp f)}.
  \end{equation}
\end{lem}

\begin{proof}
  We consider a transference plan defined by  $\pi:=(P_1\times P_2)\#f$. One can check
  that this measure has marginals $P_1 \# f$, $P_2 \# f$. Then,
  \begin{align*}
    W_1(P_1 \# f, P_2 \# f)
    \leq &\,
    \int_{\R^d \times \R^d} \abs{x-y} \pi(x,y) \,dx \,dy
    \\
    =&\,
    \int_{\R^d} \abs{P_1(x) - P_2(x)}\, f(x) \,dx
    \leq
    \norm{P_1 - P_2}_{L^\infty(\supp f)},
  \end{align*}
  which proves the lemma.
\end{proof}

\begin{lem}[Continuity with respect to time]
  \label{lem:time-W1-continuity}
  Take $T > 0$ and a field $E: [0,T] \times \R^d \to \R^d$ in the
  conditions of Hypothesis \ref{hyp:E-conditions}. Take also a measure
  $f$ on $\R^d \times \R^d$ with compact support contained in the ball
  $B_R$.

  Then, there exists $C > 0$ depending only on $\alpha$, $\beta$, $R$
  and $\norm{E}_{L^\infty([0,T] \times B_R)}$ such that
  $$W_1(\calT_E^s\#f,\calT_E^t\#f) \leq C \abs{t-s},
  \quad \text{ for any } t,s\in[0,T].$$
\end{lem}

\begin{proof}
  From Lemma \ref{lem:same-f0} and the continuity of characteristics
  with respect to time, Lemma \ref{lem:reg-chars-time}, we get
  \begin{equation*}
    W_1(\calT_E^s\#f,\calT_E^t\#f)
    \leq
    \|\calT_E^s-\calT_E^t\|_{L^\infty(\supp f)}
    \leq
    C \abs{t-s},
  \end{equation*}
  for some $C > 0$ which depends only on the quantities in the lemma.
\end{proof}

\begin{lem}
  \label{lem:same-T_[f]}
  Take a locally Lipschitz map $\calT : \R^d \to \R^d$ and $f,g \in
  \P_1(\R^d)$, both with compact support contained in the ball
  $B_R$. Then,
  \begin{equation}
    \label{eq:same-T_[f]}
    W_1(\calT \# f, \calT \# g)
    \leq
    L\,W_1(f,g),
  \end{equation}
  where $L$ is the Lipschitz constant of $\calT$ on the ball $B_R$.
\end{lem}

\begin{proof}
  Set $\pi$ to be an optimal transportation plan between $f$ and $g$. The
  measure $\gamma=(\calT \times \calT) \# \pi$ has marginals $\calT \# f$ and $\calT \# g$,
  as can be easily checked, so we can use it to bound $W_1(\calT \# f, \calT \# g)$:
  \begin{align*}
    W_1(\calT \# f, \calT \# g)
    \leq &\,
    \int_{\R^d \times \R^d} \!\!\!\!\abs{z-w} \, \gamma(z,w) \,dz \,dw
    =
    \int_{\R^d \times \R^d} \!\!\!\!\abs{\calT(z)-\calT(w)} \pi(z,w) \,dz \,dw
    \\
    \leq &\,
    L \int_{\R^d \times \R^d} \abs{z-w} \pi(z,w) \,dz \,dw
    =
    L\, W_1(f,g),
  \end{align*}
  using that the support of $\pi$ is contained in $B_R \times B_R$, as
  both $f$ and $g$ have support inside $B_R$.
\end{proof}

Recalling that $E[f]:=\nabla U\ast \rho$, the properties of
convolution lead immediately to the following information:

\begin{lem}
  \label{lem:lipschitz-field}
  Take a potential $U:\R^d \to \R$ in the conditions of Theorem
  \ref{thm:Existence}, and a measure $f \in \Po(\R^d
  \times \R^d)$ with support contained in a ball $B_R$. Then,
  \begin{equation}
    \label{eq:E-bounded}
    \norm{E[f]}_{L^\infty(B_R)}
    \leq
    \norm{\nabla U}_{L^\infty(B_{2R})},
  \end{equation}
  and
  \begin{equation}
    \label{eq:E-Lipschitz}
    \Lip_R(E[f]) \leq \Lip_{2R}(\nabla U).
  \end{equation}
\end{lem}

\begin{lem}
  \label{lem:field-W1-Linfty}
  For $f,g \in \P_1(\R^d \times \R^d)$ and $R > 0$ it holds that
  \begin{equation}
    \label{eq:field-W1-Linfty}
    \norm{E[f] - E[g]}_{L^\infty(B_R)}
    \leq \mathrm{Lip}_{2R}(\nabla U) W_1(f,g).
  \end{equation}
\end{lem}

\begin{proof}
  Take $\pi$ to be an optimal transportation plan between the
  measures $f$ and $g$. Then, for any $x \in B_R$, using that $\pi$
  has marginals $f$ and $g$,
  \begin{align*}
    E[f](x)&\, - E[g](x)
    =
    \int_{\R^d} (\rho[f](y) - \rho[g](y)) \nabla U(x-y) \,dy
    \\
    = &\,
    \int_{\R^d\times\R^d} f(y,v) \nabla U(x-y) \,dy \,dv
    - \int_{\R^d\times\R^d} g(z,w) \nabla U(x-z) \,dz \,dw
    \\
    = &\,
    \int_{\R^{4d}} (\nabla U(x-y) - \nabla U(x-z))
    \,d \pi(y,v,z,w).
  \end{align*}
  Taking absolute value,
  \begin{align*}
    \abs{E[f](z) - E[g](z)}
    \leq &\,
    \int_{\R^{4d}} |\nabla U(x-y) - \nabla U(x-z)| \,d \pi(y,v,z,w)\\
    \leq &\,
    \Lip_{2R}(\nabla U)
    \int_{\R^{4d}} \abs{y - z} \,d \pi(y,v,z,w)
    \leq
    \Lip_{2R}(\nabla U) W_1(f,g),
  \end{align*}
  using that $\pi(y,v,z,w)$ has support on $B_R \times B_R \subseteq
  \R^{4d}$.
\end{proof}

We can now give the proof of the existence and uniqueness result.

\begin{proof}[Proof of theorem \ref{thm:Existence}]
  Take $f_0 \in \P^1(\R^d \times \R^d)$ with support contained in a
  ball $B_{R^0} \subseteq \R^d \times \R^d$, for some $R^0 > 0$. We
  will prove local existence and uniqueness of solutions by a
  contraction argument in the metric space $\mathcal{F}$ formed by all
  the functions $f \in \mathcal{C}([0,T], \P_1(\R^d \times \R^d))$
  such that the support of $f_t$ is contained in $B_R$ for all $t
  \in [0,T]$, where $R := 2 R^0$ and $T > 0$ is a fixed number to be
  chosen later. Here, we consider the distance in $\mathcal{F}$ given
  by
  \begin{equation}
    \label{eq:distance-F}
    \mathcal{W}_1(f,g) := \sup_{t \in [0,T]} W_1(f_t,g_t).
  \end{equation}

  Let us define an operator on this space for which a fixed point will
  be a solution to the swarming equation \eqref{eq:swarming}. For $f
  \in \mathcal{F}$, consider $E[f] := \nabla U * \rho[f]$. Then,
  $E[f]$ satisfies Hypothesis \ref{hyp:E-conditions} (because of the
  above two Lemmas \ref{lem:lipschitz-field} and
  \ref{lem:field-W1-Linfty}, and the bound (\ref{eq:U-growth2}) on
  $\nabla U$) and we can define
  \begin{equation}
    \label{eq:def-Gamma}
    \Gamma[f](t) := \calT_{E[f]}^t \# f_0.
  \end{equation}
  In other words, $\Gamma[f]$ is the solution of the swarming
  equations obtained through the method of characteristics, with field
  $E[f]$ assumed known, and with initial condition $f_0$ at $t=0$.

  Clearly, a fixed point of $\Gamma$ is a solution to
  eq. \eqref{eq:swarming} on $[0,T]$. In order for $\Gamma$ to be well
  defined, we need to prove that $\Gamma[f]$ is again in the space
  $\mathcal{F}$, for which we need to choose $T$ appropriately. To do
  this, observe that from eq. \eqref{eq:E-bounded} in Lemma
  \ref{lem:lipschitz-field} we have
  \begin{equation*}
    \norm{E[f]}_{L^\infty([0,T] \times B_R)}
    \leq
    \norm{\nabla U}_{L^\infty(B_{2R})} =: C_1,
  \end{equation*}
  and from point 1 in lemma \ref{lem:reg-chareqs},
  \begin{equation*}
    \abs{ \frac{d}{dt} \calT_{E[f]}^t (P) }
    \leq
    C_2,
  \end{equation*}
  for all $P \in B_{R^0} \subseteq \R^d \times \R^d$, and some $C_2 >
  0$ which depends only on $\alpha$, $\beta$, $R_0$ and
  $C_1$. Choosing any $T < R^0/C_2$ one easily sees that
  $\calT_{E[f]}^t \# f_0$ has support contained in $B_R$, for all $t
  \in [0,T]$ (recall that we set $R := 2R^0$). Then, for each $t \in
  [0,T]$, $\Gamma[f](t) \in \Po(\R^d \times \R^d)$, as follows from
  mass conservation, the support of $\Gamma[f](t)$ is contained in
  $B_R$ (we just chose $T$ for this to hold), and the function $t
  \mapsto \Gamma[f](t)$ is continuous, as shown by Lemma
  \ref{lem:time-W1-continuity}.
  %
  This shows that the map $\Gamma: \mathcal{F} \to \mathcal{F}$ is well
  defined.

  Let us prove now that this map is contractive (for which we will
  have to restrict again the choice of $T$). Take two functions $f, g
  \in \mathcal{F}$, and consider $\Gamma[f], \Gamma[g]$; we want to
  show that
  \begin{equation}
    \label{eq:Gamma-Lipschitz}
    \mathcal{W}_1(\Gamma[f], \Gamma[g]) \leq C\, \mathcal{W}_1(f, g)
  \end{equation}
  for some $0 < C < 1$ which does not depend on $f$ and $g$. Using
  \eqref{eq:distance-F} and \eqref{eq:def-Gamma},
  \begin{equation}
    \label{eq:L1}
    \mathcal{W}_1(\Gamma[f], \Gamma[g])
    = \sup_{t \in [0,T]} W_1(\calT_{E[f]}^t \# f_0 , \calT_{E[g]}^t \# f_0),
  \end{equation}
  and hence we need to estimate the above quantity for each $t \in [0,
  T)$. For $t \in [0,T]$, use lemmas \ref{lem:same-f0},
  \ref{lem:dep-chars-E} and \ref{lem:field-W1-Linfty} to write
  \begin{align*}
    W_1(\calT_{E[f]}^t \# f_0 , \calT_{E[g]}^t \# f_0) &\,\leq \norm{\calT_{E[f]}^t -
      \calT_{E[g]}^t}_{L^\infty(\supp f_0)}
    \\ &\,
    \leq C(t) \sup_{s \in [0,T]} \norm{E[f_s] - E[g_s]}_{L^\infty(B_R)}
    \\ &\,
    \leq C(t)\, L \, \sup_{s \in [0,T]} W_1(f_s, g_s)
    =
    C(t)\, L \, \mathcal{W}_1(f,g),
  \end{align*}
  where $C(t)$ is the function $(e^{C_3t}-1)/C_3$ which appears in
  eq. \eqref{eq:dep-chars-E}, for some constant $C_3$ which depends
  only on $\alpha$, $\beta$, $R$, and the Lipschitz constant $L$ of
  $\grad U$ on $B_{2R}$ (see eq. \eqref{eq:E-Lipschitz}). Clearly,
  \begin{equation}
    \label{eq:small-Lipschitz-constant}
    \lim_{t \to 0} C(t) = 0.
  \end{equation}
  With \eqref{eq:L1}, this finally gives
  \begin{equation*}
    \mathcal{W}_1(\Gamma[f], \Gamma[g])
    \leq
    C(T)\, L\,\mathcal{W}_1(f,g).
  \end{equation*}
  Taking into account \eqref{eq:small-Lipschitz-constant}, we can additionally choose
  $T$ small enough so that $C(T) L < 1$. For such $T$, $\Gamma$ is
  contractive, and this proves that there is a unique fixed point of
  $\Gamma$ in $\mathcal{F}$, and hence a unique solution $f \in
  \mathcal{F}$ of eq. \eqref{eq:swarming}.

  Finally, as mass is conserved, by usual arguments one can extend
  this solution as long as the support of the solution remains
  compact. Since in our case the growth of characteristics is bounded
  (see Lemma \ref{lem:exist_cont_charact}), one can construct a unique
  global solution satisfying \eqref{eq:f-continuous} and
  \eqref{eq:supp-f}.
\end{proof}

\subsection{Stability} \label{sec:stability}

\begin{thm}
  \label{thm:stability}
  Take a potential $U$ in the conditions of Theorem
  \ref{thm:Existence}, and $f_0$, $g_0$ measures on $\R^d \times \R^d$
  with compact support, and consider the solutions $f,g$ to
  eq. \eqref{eq:swarming} given by Theorem \ref{thm:Existence} with
  initial data $f_0$ and $g_0$, respectively.

  Then, there exists a strictly increasing smooth function
  $r(t):[0,\infty)\longrightarrow \R^+_0$ with $r(0)=1$ depending only
  on the size of the support of $f_0$ and $g_0$, such that
  \begin{equation}
    \label{eq:stability}
    W_1(f_t, g_t)
    \leq
    r(t)\, W_1(f_0, g_0),
    \quad
    t \geq 0.
  \end{equation}
\end{thm}

\begin{proof}
  Fix $T > 0$, and take $R > 0$ such that $\supp f_t$ and $\supp g_t$
  are contained in $B_R$ for $t \in [0,T]$ (which can be done thanks
  to theorem \ref{thm:Existence}). For $t \in [0,T]$, call $L_t$ the
  Lipschitz constant of $\calT_{E[g]}^t$ on $B_R$, and notice that
  from lemmas \ref{lem:reg-chars} and \ref{lem:lipschitz-field} we
  have
  \begin{equation}
    \label{eq:st-proof1}
    L_t \leq e^{C_1 t}, \qquad t \in [0,T]
  \end{equation}
  for some allowed constant $C_1 > 0$. Then we have, using lemmas
  \ref{lem:same-f0}, \ref{lem:same-T_[f]}, \ref{lem:dep-chars-E} and
  \ref{lem:field-W1-Linfty},
  \begin{align*}
    W_1(f_t,g_t)
    = &\,
    W_1(\calT_{E[f]}^t \# f_0, \calT_{E[g]}^t \# g_0)
    \\
    \leq &\,
    W_1(\calT_{E[f]}^t \# f_0, \calT_{E[g]}^t \# f_0)
    +
    W_1(\calT_{E[g]}^t \# f_0, \calT_{E[g]}^t \# g_0)
    \\
    \leq &\,
    \norm{\calT_{E[f]}^t - \calT_{E[g]}^t}_{L^\infty(\supp f_0)}
    +
    L_t\, W_1(f_0, g_0)
    \\
    \leq &\,
    C_2 \int_0^t e^{C_2(t-s)} \norm{E[f_s] - E[g_s]}_{L^\infty(B_R)} \,ds
    +
    L_t\, W_1(f_0, g_0)
    \\
    \leq &\,
    C_2 \lip_{2R}(\nabla U) \int_0^t e^{C_2(t-s)} W_1(f_s, g_s) \,ds
    +
    e^{C_1 t}\, W_1(f_0, g_0).
  \end{align*}
  Calling $C = \max\{C_1, C_2, C_2\lip_{2R}(\nabla U)\}$ and
  multiplying by $e^{-Ct}$,
  \begin{equation*}
    e^{-Ct}
    W_1(f_t,g_t)
    \leq
    C \int_0^t e^{-Cs} W_1(f_s, g_s) \,ds
    +
    W_1(f_0, g_0),
    \qquad t \in [0,T],
  \end{equation*}
  and then by Gronwall's Lemma,
  \begin{equation*}
    e^{-Ct} W_1(f_t,g_t)
    \leq
    W_1(f_0,g_0) \, e^{Ct},
    \quad t \in [0,T],
  \end{equation*}
  which proves our result. We point out that the particular rate
  function $r(t)$ can be obtained by carefully looking at the
  dependencies on time of the constants above, leading to double
  exponentials.
\end{proof}

\begin{rem}[Possible generalizations]
As in Remark \ref{rem:unboundE}, by assuming more restrictive
growth properties at infinity of the potential $U$, we may weaken
the requirements on the support of the initial data allowing $f_0$
with bounded first moment for instance. We do not follow this
strategy in the present work.
\end{rem}

\subsection{Regularity}
\label{sec:regularity}

If the initial condition for eq. \eqref{eq:swarming} is more regular
than a general measure on $\R^d \times \R^d$ one can easily prove
that the solution $f$ is also more regular. For example, if $f_0$
is Lipschitz, then $f_t$ is Lipschitz for all $t \geq 0$. We will
show this next.

\begin{lem}
  \label{lem:f-Lip}
  Take an integrable function $f_0:\R^d \times \R^d \to [0,+\infty)$,
  with compact support, and assume that $f_0$ is also Lipschitz. Take
  also a potential $U \in \mathcal{C}^2(\R^D)$.

  Consider the global solution $f$ to eq. \eqref{eq:swarming} with
  initial condition $f_0$ given by Theorem \ref{thm:Existence}. Then,
  $f_t$ is Lipschitz for all $t \geq 0$.
\end{lem}

\begin{proof}
Solutions obtained from Theorem \ref{thm:Existence} have bounded
support in velocity for all times $t > 0$, and their fields $E
\equiv E(t,x) := -\nabla U * \rho$ are Lipschitz with respect to
$x$. Hence, one can rewrite
  eq. \eqref{eq:swarming} as a general equation of the form
  \begin{equation*}
    \partial_t f + \dv (a f) = 0,
  \end{equation*}
  where $a = a(t,x,v)$ is the expression appearing in the equation,
  \begin{equation*}
    a(t,x,v) = (v, E(t,x) + (\alpha - \beta\abs{v}^2) v).
  \end{equation*}
  Then, $a$ is bounded and Lipschitz with respect to $x,v$ on the
  domain considered as the support in velocity is bounded, and
  classical results show that $f_t$ is Lipschitz for all $t \geq 0$.
\end{proof}

\section{Well-posedness for General Models}

In this section we want to show that the same results we have
obtained in the previous section are also valid, with suitable
modifications, for much more general models than
\eqref{eq:swarming}. We will start by showing the adaptation of
the strategy for the Cucker-Smale system and then, we will extend
this strategy to more general models.

\subsection{Cucker-Smale Model}
\label{sec:cucker_smale}

We will prove well-posedness in a slightly more general setting
than that of the Cucker-Smale model in section \ref{sec:bg}, being
less restrictive on the communication rate and the velocity
averaging. To be more precise, we shall consider $\xi[f](x,v,t) =
\left[ (H(x,v)) \ast f \right](x,v,t)$ as in \eqref{eq:CS}, but
for a general $H:\R^d\times\R^d\to\R^d$, for which we only assume
the following hypotheses:
\begin{hyp}[Conditions on $H$]
  \label{hyp:H-conditions}
  \begin{enumerate}
  \item $H$ is locally Lipschitz.
  \item For some $C > 0$,
    \begin{equation}\label{H3}
      |H(x,v)|\leq C(1+|x|+|v|)
      \quad \text{ for all } x,v \in \R^d.
    \end{equation}
  \end{enumerate}
\end{hyp}

Since the procedure to prove the well-posedness results to
\eqref{eq:CS} is the same we have already applied in the previous
section, we will state some of the results without proof. First of
all, fix $T > 0$ and let us introduce the system of ODE's solved by
the characteristics of \eqref{eq:CS}:
\begin{subequations}
  \label{eq:characteristicsCS}
  \begin{align}
    \label{eq:characteristicsCSX}
    \frac{d}{dt} X &= V,
    \\
    \label{eq:characteristicsCSV}
    \frac{d}{dt} V &= -\xi(t,X,V),
  \end{align}
\end{subequations}
where $\xi: [0,T] \times \R^d \times \R^d \to \R^d$ is any function
satisfying the following hypothesis:

\begin{hyp}[Conditions on $\xi$]
  \label{hyp:xi-conditions}
  \begin{enumerate}
  \item $\xi$ is continuous on $[0,T] \times \R^d \times \R^d$,
  \item For some $C > 0$,
    \begin{equation}
      \label{eq:xi-growth}
      \abs{\xi(t,x,v)} \leq C(1 + \abs{x} + \abs{v}),
      \quad \text{ for all } t,x,v \in [0,T] \times \R^d \times \R^v, and
    \end{equation}
  \item $\xi$ is \emph{locally Lipschitz with respect to $x$ and $v$}, i.e., for
    any compact set $K \subseteq \R^d\times\R^d$ there is some $L_K > 0$ such
    that
    \begin{equation}
      \label{eq:Lipschitz-resp-x2}
      \abs{\xi(t,P_1) - \xi(t,P_2)} \leq L_K \abs{P_1-P_2},
      \qquad t \in [0,T], \quad P_1,P_2 \in K.
    \end{equation}
  \end{enumerate}
\end{hyp}
Under these conditions, we may consider the flow map $P_\xi^t =
P_\xi^t(x,v)$ associated to \eqref{eq:characteristicsCS}, defined as
the solution to the system \eqref{eq:characteristicsCS} with initial
condition $(x,v)$. For ease of notation, we will write the system
\eqref{eq:characteristicsCS} as
$$
\frac{dP_\xi^t}{d t}=\Psi_{\xi}(t,P_\xi^t).
$$

\begin{rem}\label{rem:suplip}
  Under Hypothesis \ref{hyp:H-conditions} on $H$, note that whenever $\tf\in
  C([0,T],\Po(\R^d\times\R^d))$ is a given compactly supported measure
  with $\supp(\tf_t)\subset B_{R^x}\times B_{R^v}$ for all $t\in
  [0,T]$, the field $\xi[\tf]=H\ast\tf$ satisfies Hypothesis
  \ref{hyp:xi-conditions}.
\end{rem}

\begin{dfn}[Notion of Solution]
  \label{dfn:solution2}
  Take $H$ satisfying Hypothesis \ref{hyp:H-conditions}, a measure
  $f_0 \in \P_1(\R^d \times \R^d)$, and $T \in (0,\infty]$.  We say
  that a function $f:[0,T] \to \P_1(\R^d \times \R^d)$ is a solution
  of the swarming equation \ref{eq:CS} with initial condition
  $f_0$ when:
  \begin{enumerate}
  \item The field $\xi = H*f$ satisfies Hypothesis \ref{hyp:xi-conditions}. 
  \item It holds $f_t = P_\xi^t \# f_0$.
  \end{enumerate}
\end{dfn}

Now, an analogue to Lemma \ref{lem:reg-chareqs} can be stated. We
shall state this one and the following lemmas for a general $\xi$
satisfying Hypothesis \eqref{hyp:xi-conditions}.
%
\begin{lem}[Regularity of the characteristic equations]
  \label{lem:reg-chareqsCS}
  Take $T>0$,
  $\xi$ satisfying Hypothesis \eqref{hyp:xi-conditions}, $R>0$ and
  $t\in [0,T]$. Then there exist constants $C$ and $L_p$ depending on
  $\Lip_R(\xi)$ and $T$ such that
  $$|\Psi_{\xi}(P)|\leq C
  \quad \text{ for all } P \in B_R \times B_R$$
  and
  $$
  |\Psi_{\xi}(P_1)-\Psi_{\xi}(P_2)|\leq L_p|P_1-P_2|
  \quad \text{ for all } P_1, P_2 \in B_R \times B_R.
  $$
\end{lem}

Lemmas \ref{lem:dep-chareqs-E}--\ref{lem:reg-chars-time} are valid
as they are presented, taking $\xi$ and Hypothesis
\ref{hyp:xi-conditions} to play the role of $E$ and Hypothesis
\ref{hyp:E-conditions}, and making the obvious minor modifications
on the dependence of the constants. Now we can look at the
existence of solutions:

\begin{thm}[Existence and uniqueness of measure solutions]
  \label{thm:Existence-CS}
  Assume $H$ satisfies Hypothesis \ref{hyp:H-conditions}, and take
  $f_0\in \Po(\R^d\times\R^d)$ compactly supported. Then there exists
  a unique solution $f\in C([0,T],\Po(\R^d\times\R^d))$ to equation
  \eqref{eq:CS} in the sense of Definition \ref{dfn:solution2} with
  initial condition $f_0$. Moreover, the solution remains compactly
  supported for all $t\in [0,T]$, i.e., there exist $R^x$ and $R^v$
  depending on $T$, $H$ and the support of $f_0$, such that
  $$\supp(f_t)\subset B_{R^x}\times B_{R^v}
  \text{ for all } t\in [0,T].$$
\end{thm}

The proof of this result can be done following the same steps as for
proving Theorem \ref{thm:Existence}. Lemmas \ref{lem:same-f0} to
\ref{lem:same-T_[f]} still hold in this situation, and we recombine
Lemmas \ref{lem:lipschitz-field} and \ref{lem:field-W1-Linfty} in the
following result:

\begin{lem}
  \label{lem:lipschitz-W1-Linfty-field-CS}
  Take $H$ satisfying Hypothesis \ref{hyp:H-conditions}, $\tf\in
  \Po(\R^d\times\R^d)$ with $\supp(\tf_t)\subset B_{R^x}\times
  B_{R^v}$, and $\xi := \xi[\tf] = H * \tf$. Then, for any
  $R>0$ $$\Lip_R(\xi) \leq \Lip_{R+\hat{R}}(H),$$ with
  $\hat{R}:=\max{R^x,R^v}$. Furthermore, if $\tilde g \in
  \Po(\R^d\times\R^d)$ it holds that
  \begin{equation}
    \label{eq:field-W1-Linfty-cs}
    \norm{\xi[\tf] - \xi[\tilde g]}_{\rmL^{\infty}(B_R)}
    \leq \Lip_{R+\hat{R}}(H)\, W_1(\tf,\tilde g).
  \end{equation}
\end{lem}

\begin{proof}
  The first part follows directly from the properties of
  convolution. For the second one, take $\pi$ to be an optimal
  transportation plan between the measures $\tf$ and $\tilde g$. Then,
  for any $x,v \in B_R$, using that $\pi$ has marginals $\tf$ and
  $\tilde g$,
  \begin{align*}
    \xi[\tf](x,v) &\,- \xi[\tilde g](x,v)\\
    \,=&
    \int_{\R^{2d}} \!\! H(x-y,v-u)\tf(y,u) \,d (y,u)\\
    &- \int_{\R^{2d}} \!\!H(x-z,v-w)\tilde g(z,w) \,d (z,w)
    \\
    \,=&
    \int_{\R^{4d}} \left[H(x-y,v-u) - H(x-z,v-w)\right]
    \,d \pi(y,u,z,w).
  \end{align*}
  Taking absolute value, and using that the support of $\pi$ is
  contained in the ball $B_{\hat R} \subseteq \R^{4d}$,
  \begin{multline*}
    |\xi[\tf](x,v) - \xi[\tilde g](x,\,v)|
    \\
    \leq
    \int_{B_{\hat R}} \!\! |H(x-y,v-u) - H(x-z,v-w)| \,d \pi(y,u,z,w)
    \\
    \leq \,
    \Lip_{R+\hat{R}}(H)
    \int_{\R^{4n}} \!\! \abs{(y - z, u-w)} \,d \pi(y,u,z,w)
    =
    \Lip_{R+\hat{R}}(H) W_1(\tf,\tilde g).
  \end{multline*}
\end{proof}

Finally, a stability result also follows using the same steps as
in Theorem \ref{thm:stability}.

\begin{thm}[Stability in $W_1$]
  \label{thm:stability-CS} Assume $H$ satisfies Hypothesis
  \ref{hyp:H-conditions}, and $f_0,g_0\in \Po(\R^d\times\R^d)$ are
  compactly supported. Consider the solutions $f,g$ to
  eq. \eqref{eq:CS} given by Theorem \ref{thm:Existence-CS} with
  initial data $f_0$ and $g_0$, respectively. Then, there exists a
  strictly increasing function $r(t):[0,\infty)\longrightarrow \R^+_0$
  with $r(0)=1$ depending only on $H$ and the size of the support of
  $f_0$ and $g_0$, such that
  \begin{equation}
    \label{eq:stability2}
    W_1(f_t, g_t)
    \leq
    r(t)\, W_1(f_0, g_0),
    \quad
    t \geq 0.
  \end{equation}
\end{thm}

\begin{remark}[Evolution of the support in the Cucker-Smale model]
  In \cite{cfrt09} it is shown a sharp bound on the evolution of the
  support for the kinetic Cucker-Smale equation in which
  $H(x,v)=w(x)v$. More precisely, it is proved that for any given $f_0
  \in \Po(\R^d\times \R^d)$ compactly supported, we have that
  $$
  \supp(f_t) \subset B(x_c(0)+m t,R^x(t)) \times B(m,R^v(t)),
  $$
  with
  $$
  R^x(t) \leq \bar R \qquad \mbox{and} \qquad R^v(t) \leq R_0\,
  e^{-\lambda t}
  $$
  for some $\bar R$ depending only on
  $R_0=\max\{R^x(0),R^v(0)\}$ and $\lambda=w(2\bar R)$. Here, $m$
  stands for the mean velocity of the system
$$
m:=\int_{\R^{2d}} v\, f(t,x,v)\,d x \,d v,
$$
which is preserved along its evolution. This precise bound on the
support and the particular choice of $H$ lead to a uniform control in
time of the constants $\Lip_R(H)$ and $L_p$ in the results above,
which are now bounded for all times. A tedious but straightforward
computation leads to a rate $r(t)$ in the stability result which is
exponentially increasing. Indeed, if we follow the steps of the proof
of Theorem \ref{thm:stability} for the particular case of the
Cucker-Smale model we can see that, since $R_x$ and $R_v$ are not
increasing with time, the numbers $C_1$ and $C_2$ that appear there
can be chosen independently of time, whence $r(t)$ shall grow at most
exponentially.
\end{remark}

\begin{remark}[Comparison with Literature]
As already mentioned above, the particular case of the kinetic
Cucker-Smale model has already been approached in \cite{HL08}
where the authors give a well-posedness result based on the
bounded Lipschitz distance. Here, we recover the same result but
based on the stability in the Wasserstein distance $W_1$, which
allows us to obtain sharper constants and rates.
\end{remark}

\subsection{General Models}
\label{sec:general_models}

With the techniques used in the previous sections one can include
quite general kinetic models in the well-posedness theory. In this
section we illustrate this by giving a result for a model which
includes both the potential interaction and self-propulsion
effects of section \ref{sec:swarming_model}, the
velocity-averaging effect of section \ref{sec:cucker_smale} and
the more general models above \cite{LLE,LLE2}.

Let us introduce some notation for this section: $\Pc(\R^d\times\R^d)$
denotes the subset of $\Po(\R^d\times\R^d)$ consisting of measures of
compact support in $\R^d\times\R^d$, and we consider the non-complete
metric space ${\cal A}:=\mathcal{C}([0,T], \Pc(\R^d\times\R^d))$
endowed with the distance ${\cal W}_1$. On the other hand, we consider
the set of functions ${\cal B}:=\mathcal{C}([0,T],\lip_{loc}
(\R^d\times\R^d,\R^d))$, which in particular are locally Lipschitz
with respect to $(x,v)$, uniformly in time. We consider an operator
${\cal H}[\cdot]:{\cal A}\longrightarrow {\cal B}$ and assume the
following:
\begin{hyp}[Hypothesis on a general operator]
  \label{hyp:general-model}
  Take any $R_0 > 0$ and $f, g \in \mathcal{A}$ such that $\supp(f_t)
  \cup \supp(g_t) \subseteq B_{R_0}$ for all $t \in [0,T]$. Then for
  any ball $B_R \subset \R^d \times \R^d$, there exists a constant $C
  = C(R,R_0)$ such that
  \begin{gather*}
    \max_{t\in [0,T]}
    \|{\cal H}[f]-{\cal H}[g]\|_{L^\infty(B_R)}
    \leq
    C \,{\cal W}_1(f,g),
    \\
    \max_{t \in [0,T]} \lip_R({\cal H}[f])
    \leq C .
  \end{gather*}
\end{hyp}

Associated to this operator, we can consider the following general
equation:
\begin{equation}
  \label{eq:general-model}
  \partial_t f + v \cdot \grad_x f
  - \nabla_v \cdot [ {\cal H}[f] f ]
  = 0.
\end{equation}

\begin{remark}[Generalization]
  It is not difficult to see that the choices ${\cal
    H}[f]=(\alpha-\beta|v|^2)v-\nabla U * \rho$ and ${\cal H}[f]=H *
  f$ correspond to \eqref{eq:swarming} and \eqref{eq:CS},
  respectively, and that they satisfy Hypothesis
  \ref{hyp:general-model} if we assume the hypotheses of Theorems
  \ref{thm:Existence} and \ref{thm:Existence-CS} respectively.
  Moreover, one can cook up an operator of the form:
  $$
  {\cal H}[f]=F_A(x,v)+G(x) * \rho + H(x,v) * f
  $$
  with $F_A$, $G$ and $H$ given functions satisfying suitable
  hypotheses, such that the kinetic equation \eqref{eq:general-model}
  corresponds to the model \eqref{eq:lle}.
\end{remark}

We will additionally require the following:
\begin{hyp}[Additional constraint on {\cal H}]
  \label{hyp:general-model2}
  Given $f \in \mathcal{C}([0,T], \Pc(B_{R_0}))$, and for any
  initial condition $(X^0, V^0) \in \R^d \times \R^d$, the following
  system of ordinary differential equations has a globally defined
  solution:
  \begin{subequations}
  \label{eq:characteristics-general}
  \begin{align}
    \label{eq:characteristicsX-general}
    \frac{d}{dt} X &= V,
    \\
    \label{eq:characteristicsV-general}
    \frac{d}{dt} V &= {\cal H}[f](t,X,V) ,
    \\
    X(0) &= X^0, \quad V(0) = V^0.
  \end{align}
\end{subequations}
\end{hyp}
Of course, this is a requirement that has to be checked for every
particular model, and it is difficult to give useful properties of
${\cal H}$ that imply this and are general enough to encompass a
range of utile models; therefore, we prefer to give a general
condition which reduces the problem of existence and stability to
the simpler one of existence of the characteristics.

In the above conditions one can follow a completely analogous argument
to that in the proof of Theorems \ref{thm:Existence} and
\ref{thm:stability}, and obtain the following result:

\begin{thm}[Existence, uniqueness and stability of measure solutions for a
  general model]
  \label{thm:Existence-general}
  Take an operator ${\cal H}[\cdot]:{\cal A}\longrightarrow {\cal B}$
  satisfying Hypotheses \ref{hyp:general-model} and
  \ref{hyp:general-model2}, and $f_0$ a measure on $\R^d \times \R^d$
  with compact support. There exists a solution $f$ on $[0,+\infty)$
  to equation \eqref{eq:general-model} with initial condition
  $f_0$. In addition,
  \begin{equation}
    \label{eq:f-continuous-general}
    f \in \mathcal{C}([0,+\infty); \Pc(\R^d\times \R^d))
  \end{equation}
  and there is some increasing function $R = R(T)$ such that for all $T>0$,
  \begin{equation}
    \label{eq:supp-f-general}
    \supp f_t \subseteq B_{R(T)} \subseteq \R^d \times \R^d
    \quad \text{ for all } t \in [0,T].
  \end{equation}
  This solution is unique among the family of solutions satisfying
  \eqref{eq:f-continuous-general} and \eqref{eq:supp-f-general}.

  Moreover, given any other initial data $g_0\in \Pc(\R^d\times\R^d)$ and $g$
  its corresponding solution, then there exists a
  strictly increasing function $r(t):[0,\infty)\longrightarrow \R^+_0$
  with $r(0)=1$ depending only on ${\cal H}$ and the size of the support of
  $f_0$ and $g_0$, such that
  \begin{equation*}
    W_1(f_t, g_t)
    \leq
    r(t)\, W_1(f_0, g_0),
    \quad
    t \geq 0.
  \end{equation*}
\end{thm}

\section{Consequences of Stability}
\subsection{$N$-Particle approximation and the mean-field limit}
\label{sec:meso}

The stability theorems \ref{thm:stability} and
\ref{thm:stability-CS}, or the general version
\ref{thm:Existence-general}, give in particular a justification of
the approximation of this family of models by a finite set of
particles satisfying a system of ordinary differential equations.
We will state results for the general model
\eqref{eq:general-model}, under the conditions on $\cal H$ from
section \ref{sec:general_models}.

One can easily check that the following holds:

\begin{lem}[Particle solutions]
  \label{lem:ode-is-measure-solution}
  Assume $\cal H$ satisfies the conditions of Theorem
  \ref{thm:Existence-general}.  Take $N$ positive numbers
  $m_1,\dots,m_N$, and consider the following system of differential
  equations:
  \begin{subequations}
    \label{eq:swarming-odes}
    \begin{align}
      &\dot{x_i} = v_i,
      \quad &i=1,\dots,N,
      \\
      &\dot{v_i} = \sum_{j \neq i} m_j {\cal H}[f^N](t,x_i,v_i),
      \quad &i=1,\dots,N,
    \end{align}
  \end{subequations}
  where $f^N:[0,T] \to \P_1(\R^d \times \R^d)$ is the measure defined
  by
  \begin{equation}
    \label{eq:ode-equivalent-measure}
    f^N_t := \sum_{i=1}^N m_i\, \delta_{(x_i(t), v_i(t))}.
  \end{equation}
  If $x_i,v_i:[0,T] \to \R^d$, for $i = 1,\dots,N$, is a solution to
  the system \eqref{eq:swarming-odes}, then the function $f^N$ is the
  solution to \eqref{eq:general-model} with initial condition
  \begin{equation}
    \label{eq:ode-equivalent-measure-initial}
    f^N_0 = \sum_{i=1}^N m_i\, \delta_{(x_i(0), v_i(0))}.
  \end{equation}
\end{lem}

As a consequence of the stability in $W_1$, we have an alternative
method to derive the kinetic equations \eqref{eq:swarming},
\eqref{eq:CS} or \eqref{eq:general-model}, based on the
convergence of particle approximations, other than the formal
BBGKY hierarchy in \cite{Bo,CDP}.

\begin{cor}[Convergence of the particle method]
  \label{cor:N-particle}
  Given $f_0\in \Po(\R^d\times\R^d)$ compactly supported and ${\cal
    H}$ satisfying the conditions of Theorem
  \ref{thm:Existence-general}, take a sequence of $f_0^N$ of measures
  of the form \eqref{eq:ode-equivalent-measure-initial} (with $m_i$,
  $x_i(0)$ and $v_i(0)$ possibly varying with $N$), in such a way that
  $$
  \lim_{N\to \infty} W_1(f_0^N, f_0) = 0.
  $$
  Consider $f^N_t$ given by \eqref{eq:ode-equivalent-measure}, where
  $x_i(t)$ and $v_i(t)$ are the solution to system
  \eqref{eq:swarming-odes} with initial conditions $x_i(0)$,
  $v_i(0)$. Then,
  $$
  \lim_{N\to \infty} W_1(f_t^N, f_t) = 0 ,
  $$
  for all $t\geq 0$, where $f = f(t,x,v)$ is the unique measure
  solution to eq. \eqref{eq:general-model} with initial data $f_0$.
\end{cor}

\subsection{Hydrodynamic limit}
\label{sec:hydro}

We state our hydrodynamic limit result for eq.
\eqref{eq:swarming}. If we look for solutions of
(\ref{eq:swarming}) of the form
\begin{equation}
  \label{eq:f-velocity-delta}
  f(t,x,v) = \rho(t,x)\, \delta(v - u(t,x))
\end{equation}
for some functions $\rho, u: [0,T] \times \R^d \to \R$, one
formally obtains that $\rho$ and $u$ should satisfy the following
equations:
\begin{subequations}
  \label{eq:hydro}
  \begin{align}
    &\partial_t \rho + \dv_x(\rho u) = 0,\\
    &\partial_t u + (u\cdot \nabla) u = u (\alpha - \beta \abs{u}^2) - \grad U *
      \rho .
  \end{align}
\end{subequations}
This is made precise by the following result whose existence part
was already obtained in \cite{CDP}:

\begin{lem}[Uniqueness for Hydrodynamic Solutions]
  \label{lem:existence-hydro}
Take a potential $U \in \mathcal{C}^2(\R^d)$ and assume that there
exists a smooth solution $(\rho,u)$ with initial data
$(\rho_0,u_0)$ to the system \eqref{eq:hydro} defined on the
interval $[0,T]$. Then, if we define $f:[0,+\infty) \to
\P_1(\R^d\times\R^d)$ by
  \begin{equation}
    \label{eq:hydro-measure-equivalent}
    \int_{\R^d\times\R^d} f(t,x,v)\, \phi(x,v) \,dx\,dv
    = \int_{\R^d} \phi(x,u(t,x))\, \rho(t,x) \,dx
  \end{equation}
for any test function $\phi \in {\cal C}^0_C(\R^d\times\R^d)$,
then $f$ is the unique solution to \eqref{eq:swarming} obtained
from Theorem \ref{thm:Existence} with initial condition
$f_0=\rho_0\delta(v-u_0)$.
\end{lem}

As a direct consequence of Lemma \ref{lem:existence-hydro} and the
stability result in Theorem \ref{thm:stability}, we get the
following result.

\begin{cor}[Local-in-time Stability of Hydrodynamics]
Take a potential $U \in \mathcal{C}^2(\R^d)$ and assume that there
exists a smooth solution $(\rho,u)$ with initial data
$(\rho_0,u_0)$ to the system \eqref{eq:hydro} defined on the
interval $[0,T]$. Let us consider a sequence of initial data
$f_0^k \in \P_1(\R^d\times\R^d)$ such that
$$
\lim_{k\to \infty} W_1(f_0^k, \rho_0\, \delta(v-u_0)) = 0.
$$
Consider the solution $f^k$ to the swarming eq. \eqref{eq:swarming}
with initial data $f_0^k$. Then,
$$
\lim_{k\to \infty} W_1(f^k_t, f_t) = 0 ,
$$
for all $t\in [0,T]$ with $f(t,x,v)=\rho(t,x)\, \delta(v -
u(t,x))$.
\end{cor}


\section*{Acknowledgments}
The authors acknowledge support from the project MTM2008-06349-C03-03
DGI-MCI (Spain) and 2009-SGR-345 from AGAUR-Generalitat de
Catalunya. This work was developed at the CRM-Barcelona during the
thematic program in ``Mathematical Biology'' in 2009.


\end{document}